\documentclass{article}
\usepackage{amssymb,amsmath,amsthm, hyperref}
\usepackage[latin1]{inputenc}
\newtheorem{theorem}{Theorem}
\newtheorem{definition}{Definition}
\newtheorem{lemma}{Lemma}
\newtheorem{proposition}{Proposition}

\date{}
\numberwithin{equation}{section} \numberwithin{theorem}{section}
\numberwithin{lemma}{section} \numberwithin{corollary}{section}
\numberwithin{remark}{section} \numberwithin{proposition}{section}
\numberwithin{definition}{section}

\newcommand{\Div}{\operatorname{div}}
\newcommand{\loc}{\operatorname{loc}}
\newcommand{\osc}{\operatorname{osc}}

\begin{document}

\title{Optimal design problems with \\ fractional
diffusions}

\author{Eduardo V. Teixeira\footnote{Universidade Federal do Cear\'a, Av. Humberto Monte s/n Campus of Pici - Bloco 914, Fortaleza - Cear\'a - Brazil 60.455-760 -- {\tt teixeira@mat.ufc.br}}  $ \ \& $ Rafayel Teymurazyan\footnote{Universidade de Coimbra, Departamento de Matem\'atica, 3001-501 Coimbra, Portugal -- {\tt rafayel@utexas.edu}}}

\maketitle

\begin{abstract}
In this article we study optimization problems ruled by fractional diffusion operators with volume constraints. By means of penalization techniques we prove existence of solutions. We also show that every solution is locally of class $C^{0,\alpha}$ (optimal regularity), and that the free boundary is a $C^{1,\gamma}$ surface, up to a $\mathcal{H}^{n-1}$-negligible set.
\end{abstract}

\section{Introduction}
The goal of this paper is to establish existence and geometric properties of nonlocal optimal design problems. This class of problems arises in the study of best insulation devices. Motivations also come from semi-conductor theory, plasma physics, flame propagations etc. When taking into account eventual turbulence or long-range integrations, the model becomes more accurate when ruled by nonlocal operators, such as $(-\Delta)^\alpha$. In particular, the model becomes sensible to interior changes in the temperature; local versions of the problem can only feel changes on the boundary of the body to be insulated.

Let us recall that the fractional Laplacian $(-\Delta)^\alpha$ is given by
\begin{equation*}
(-\Delta)^\alpha u(x)=C_{n,\alpha}\textrm{PV}\int_{\mathbb{R}^n}\frac{u(x)-u(y)}{|x-y|^{n+2\alpha}}\,dy,
\end{equation*}
where $\textrm{PV}$ is the Cauchy principal value and $C_{n,\alpha}$ is a normalization constant. The free boundary optimization problem we study here takes the following formulation:  given a smooth domain $D \subset \mathbb{R}^n$, a nonnegative function $\varphi \colon D \to \mathbb{R}$,  and a positive number $\omega>0$,   minimize the $\alpha$-energy
\begin{equation}\label{alpha energy}
	 J_\alpha(u):=\int_{\mathbb{R}^n}\int_{\mathbb{R}^n}\frac{|u(x)-u(y)|^2}{|x-y|^{n+2\alpha}}\,dx\,dy
\end{equation}
among competing functions $u$ within the functional set
\begin{equation}\label{set}
	K(\alpha,\omega, \varphi):=\left \{ u\in H^{\alpha}(\mathbb{R}^n) \; \big  |  \; u = \varphi \text{ in } D \text{ and } \mathcal{L}^n(\{u>0\} \cap D^c)=\omega \right \},
\end{equation}
where  $\mathcal{L}^n$ is the $n$-dimensional Lebesgue measure, and $H^{\alpha}(\mathbb{R}^n)$ is the $\alpha$-fractional Sobolev space (see, for example, \cite{DnPV12}), i.e. the set of functions $u$ for which
$$
	 \|u\|_{H^\alpha}=\sqrt{\int_{\mathbb{R}^n}(1+|\xi|^{2\alpha})|\hat{u}(\xi)|^\alpha\,d\xi}<\infty.
$$
By usual methods in the calculus of variations it is hard to perform volume preserving perturbations as to derive energy estimates. In turn, proving existence of a minimizer $u$ as well as regularity properties of $u$ and its free boundary $\partial\{u>0\}\cap D$  is in general difficult tasks from the mathematical view point.

Local versions of the problem have been well studied in the literature,  see  \cite{AAC86, ACS88, BMW06, OT06, T05, T10} among others. A celebrate approach for tackling problems involving volume constrains is based on penalization techniques. The idea is to introduce an artificial parameter in the energy functional which charges for configurations that exceed the volume budget. For fixed values of the penalization parameter, the penalized functional can then be analyzed by free boundary variational methods. Nonlocal free boundary variational tools were first introduced in \cite{CRS10}. Thus, the starting point of this current paper is to obtain regularity results for minimizers of nonlocal free boundary problems with fixed penalized parameter. This requires adjustments of existing free boundary methods from \cite{CRS10}; some extra work is needed though. Of course all analytic and geometric estimates obtained depend upon the penalized parameter, and they blow-up as the penalization term goes to infinity.

The auxiliary penalization problem we consider here takes the following set-up. Fix an $\varepsilon>0$, we define the $\varepsilon$-energy functional
\begin{equation}\label{alpha energy esp}
J_\varepsilon(v):=\int_{\mathbb{R}^n}\int_{\mathbb{R}^n}\frac{|u(x)-u(y)|^2}{|x-y|^{n+2\alpha}}\,dx\,dy+f_\varepsilon \left (\mathcal{L}^n(\{u>0\} \cap D^c) \right ),
\end{equation}
where
$$
f_\varepsilon(s):=
\left\{
\begin{array}{c}
\frac{1}{\varepsilon}(s-\omega)\,\textrm{ for }s\geq\omega,\\
\varepsilon(s-\omega)\,\,\textrm{ for }s\leq\omega.
\end{array}
\right.
$$
Define the functional set
\begin{equation}\label{set w/o vol}
	K:=\left \{u\in H^\alpha(\mathbb{R}^n) \; \big | \;    u=\varphi \textrm{ in }  D \right \},
\end{equation}
and then, the penalized problem becomes
\begin{equation}\label{1.4}
\textrm{find }\,u\in K\,\textrm{ such that }J_\varepsilon(u)=\inf_{v\in K}J_\varepsilon(v).
\end{equation}

As previously mentioned, the minimization problem \eqref{1.4} is similar to the one treated in \cite{CRS10}. However,  {\it there's no such thing as a free lunch} -- the key and in general hard issue, though, is to prove that the aimed volume is attained for small (but still positive) values of the penalization parameter. This is only possible by means of a refined control on the rate between volume decreasing versus energy increasing, of competing shapes. The appropriate tool for such a control is the so-called Hadamard's variational formula, whose local (smooth) version is known for over one-hundred years, \cite{Had}. Hence, one of the main difficulties we handle in this work is to derive a measure-theoretic, $\alpha$-fractional Hadamard's formula for domain variations. This is accomplished in section \ref{sct HF}.

Another difficulty in dealing with nonlocal optimal design problems is the lack of {\it local} information on the unknowns. To bypass this inconvenience, we will make use of the extension property discovered in \cite{CS07}. We consider the upper half-space
$$
	 \mathbb{R}_+^{n+1}=\{(x,y)\in\mathbb{R}^n\times\mathbb{R}_+\},
$$
and set $\beta=1-2\alpha$. For $u\in C^2(\mathbb{R}^n)$ we solve the Dirichlet problem
\begin{eqnarray}\label{1.5}
  -\Div(y^\beta\nabla v)&=&0\,\textrm{ in }\mathbb{R}^{n+1}_+, \\
  v(x,0)&=&u(x).\nonumber
\end{eqnarray}
A solution to such a problem can be obtained by convolution with the Poisson kernel $P_{n,\alpha}(x,y)$ of the operator $\Div(y^\beta\nabla)$ in $\mathbb{R}^{n+1}_+$, see \cite{CS07}. We have
$$
P_{n,\alpha}(x,y)=q_{n,\alpha}\frac{y^{2\alpha}}{(x^2+y^2)^{(n+2\alpha)/2}},
$$
where $q_{n,\alpha}$ is such that $\int P_{n,\alpha}(x,1)\,dx=1$. From \cite{CS07} we also know that
\begin{theorem}\label{t1.1}
We have $(-\Delta)^\alpha u(x)=-\displaystyle{\lim_{y\rightarrow0}y^\beta v_y(x,y)}$.
\end{theorem}
Because of the divergence form of the elliptic operator in \eqref{1.5}, a Dirichlet integral is available. We also observe that if $u$ solves \eqref{1.5} and $\partial_y u = 0$ on the hyperplane $\{y=0\}$, one can extend it evenly across  $\{y=0\}$, and the new equation satisfied by $u$ is
\begin{eqnarray}\label{1.6}
  -\Div(|y|^\beta\nabla v)&=&0\,\textrm{ in }\mathbb{R}^{n+1},\nonumber \\
  v(x,0)&=&u(x).\nonumber
\end{eqnarray}
For any open subset $\Omega$ of $\mathbb{R}^{n+1}$, we introduce the weighted Hilbert space
$$
H^1(\beta,\Omega):=\big\{u\in L^2(\Omega_+);\, |y|^{\beta/2}\nabla u\in L^2(\Omega)\big\},
$$
where $\Omega_+=\Omega\cap\mathbb{R}_+^{n+1}$. We set
\begin{equation} \label{I esp}
	 I_\varepsilon(u,\Omega):=\int_\Omega|y|^\beta|\nabla u|^2\,dx\,dy+f_\varepsilon(\mathcal{L}^n(\{u>0\}\cap\mathbb{R}^n\cap\Omega)),\,\,u\in H^1(\beta,\Omega).
\end{equation}
The study of penalized problem \eqref{1.4} is now replaced by the study of local minimizers of $I_\varepsilon$, i.e. functions $u$ that are in $H^1(\beta,B_1)$ and satisfy
\begin{equation}\label{1.7}
\forall B\subset B_1,\,\,\forall v\in H^1(\beta,B)\,\textrm{ with }\,v=u\,\textrm{ on }\,\partial B,\,\,\,\,\,I_\varepsilon(u,B)\leq I_\varepsilon(v,B).
\end{equation}

The paper is organized as follows: in section \ref{stc P} we list few analytic and geometric properties we will need for the study of the (fixed) penalized functional \eqref{I esp}. Still in section \ref{stc P} we show that a minimum of $I_\varepsilon$ is $\alpha$-H\"older continuous, which corresponds to the optimal regularity for the free boundary problem. We also show that it is non-degenerate. The constants though do depend on the penalized parameter $\varepsilon$ and they blow-up as $\varepsilon \to 0$. In section \ref{sct FPS} we obtain measure estimates on the free boundary. In section \ref{sct HF} we derive a fractional Hadamard's variational formula and in section \ref{sct Rec} we prove that for $\varepsilon > 0$ small enough -- but still positive -- the volume constraint is verified. This finally provides the existence of a minimum for the fractional optimization problem with volume constraint.

\section{Preliminaries}\label{stc P}

We start off by listing few properties we will need along the article. Initially, let us set, for the sake of notations, that a minimizer of \eqref{1.7} will be denoted by $u$ instead of $u_\varepsilon$. Here and afterwards by $B_r(x,y)$ we denote the ball in $\mathbb{R}^{n+1}$ centered at $(x,y)$ and of radius $r$. When $x=0$, we simply write $B_r$. We will also write $B_r^n(x)$ for the ball in $\mathbb{R}^n$ centered at $x$ with radius $r$. We now prove existence of minimizers for the $\varepsilon$-penalized problem.
\begin{lemma}\label{l2.1}
Given a smooth boundary datum, $\varphi(x,y)$ defined on $\partial B_1\cap\mathbb{R}_+^{n+1}$, problem \eqref{1.7} has an absolute minimum $u$.
\end{lemma}
\begin{proof}
Take $u_0$ with $\mathcal{L}^n(\{u_0>0\})\leq\omega$, then $I_\varepsilon(u_0)\leq C$ (uniformly in $\varepsilon$), also $I_\varepsilon\geq-\omega$. Therefore, there exists a minimizing sequence $\{u_k\}_{k\in\mathbb{N}}$. The sequence is bounded in $H^\alpha(B_1)$ and, because $H^\alpha$ is continuously embedded into $L^{2n/(n-2\alpha)}$, the sequence $\{u_k\}_{k\in\mathbb{N}}$ converges (up to a subsequence) to a function $u$ strongly in $L^{2n/(n-2\alpha) - \epsilon}$, for any $\epsilon >0$, and almost everywhere in $\mathbb{R}^n$. Thus,
$$
\mathcal{L}^n(\{u>0\})\leq\liminf_{k\rightarrow\infty}\mathcal{L}^n(\{u_k>0\})
$$
and
$$
\int_{B_1}|y|^\beta|\nabla u|^2\leq\liminf_{k\rightarrow\infty}\int_{B_1}|y|^\beta|\nabla u_k|^2.
$$
Hence, $u\in H^1(\beta,B_1)$ and since $f_\varepsilon$ is a continuous and increasing function, one has
$$
I_\varepsilon(u)\leq\liminf_{k\rightarrow\infty}I_\varepsilon(u_k)=\inf_{v\in H^1(\beta,B_1)} I_\varepsilon(v).
$$
Therefore $u$ is an absolute minimizer of $I_\varepsilon$ in $H^1(\beta,B_1)$.
\end{proof}


For the behavior of a minimizer in its positivity set, we state the following proposition and refer the reader for its proof to Proposition 3.1 of \cite{CRS10}.

\begin{proposition}\label{p1.1}
Let $u$ be a local minimizer in \eqref{1.7}, and $x_0\in\mathbb{R}^n$ be such that $u(x_0,0)>0$. Then
$$
\lim_{y\rightarrow0}|y|^\beta\nabla u(x_0,y)=0.
$$
Moreover, if $u$ is defined in $\mathbb{R}^{n+1}$, is positive outside the hyperplane $\{y=0\}$ and satisfies $\Div(|y|^\beta\nabla u)=0$ in its positivity set, together with the estimate $u(x,y)=O(|(x,y)|^\alpha)$, then $(-\Delta)^\alpha u(\cdot,0)=0$ in $\mathbb{R}^n\cap\{u>0\}$.
\end{proposition}

We now turn our attention to optimal H\"older estimates for minimizers. In the proof of the following theorem we will use the characterization of H\"{o}lder functions, \cite{M66}: given $\alpha\in(0,1)$, if $B$ is a ball in $\mathbb{R}^{n+1}$, and if there are $C>0$ and $p\in(1,n+1)$ such that
\begin{equation}\label{2.1}
\forall x\in B,\,\forall r<\textrm{dist}(x,\partial B),\,\int_{B_r(x)}|\nabla u|^p\leq Cr^{n+1-p+p\alpha},
\end{equation}
then $u\in C^{0,\alpha}(B)$.

\begin{theorem}[Optimal regularity]\label{t2.1}
If $u$ is a local minimizer of $I_\varepsilon$ posed in $B_1$, then $u\in C^{0,\alpha}_{\loc}(B_1)$.
\end{theorem}
\begin{proof}
For every $r\in(0,1)$ and $(x_0,y_0)\in B_1$, let us consider the harmonic replacement of $u$ in $B_r(x_0,y_0)$ (we have chosen $r<1-|x_0|$), i.e. the solution of
\begin{equation}\label{2.2}
-\Div(|y|^\beta\nabla h_r^{x_0,y_0})=0\,\textrm{ in }\, B_r(x_0,y_0),\,\,\,\,h_r^{x_0,y_0}=u\,\textrm{ on }\,\partial B_r(x_0,y_0).
\end{equation}
From the translation invariance in $x$, we may assume $x_0=0$. We will use the notation $h_r$ for the solution of \eqref{2.2}. Note that $u$ is an admissible Dirichlet datum (see Theorems 2.2 and 2.3 of \cite{CRS10}). Note also, that for all $r>0$ one has $I_\varepsilon(u,B_r)\leq I_\varepsilon(h_r,B_r)$. The latter implies
\begin{equation}\label{2.3}
\int_{B_r}|y|^\beta|\nabla u|^2\leq\int_{B_r}|y|^\beta|\nabla h_r|^2+Cr^n.
\end{equation}
Note that although the constant $C$ depends on $\varepsilon$, but this does not bother us, because $\varepsilon$ will always be fixed (maybe very small but fixed).

The rest of the proof is the same as of Theorem 1.1 in \cite{CRS10}. We bring it here for reader's convenience.

Due to the identity $\displaystyle{\int_{B_r}}|y|^\beta\nabla h_r\cdot\nabla(u-h_r)=0$, we get from \eqref{2.3}
$$
\int_{B_r}|y|^\beta|\nabla(u-h_r)|^2\leq Cr^n.
$$
Therefore, if $r<\rho<1$,
\begin{eqnarray}\label{2.4}
\int_{B_r}|y|^\beta|\nabla u|^2&=&\int_{B_r}|y|^\beta|\nabla(u-h_\rho+h_\rho)|^2\nonumber \\
&\leq&2\bigg(\int_{B_\rho}|y|^\beta|\nabla(u-h_\rho)|^2+\int_{B_r}|y|^\beta|\nabla h_\rho|^2\bigg)\nonumber \\
&\leq&C\rho^n+2\int_{B_r}|y|^\beta|\nabla h_\rho|^2\nonumber \\
&\leq&C\rho^n+C\bigg(\frac{r}{\rho}\bigg)^{n+1+\beta}\int_{B_\rho}|y|^\beta|\nabla h_\rho|^2\,\,\,\,\textrm{ by Theorem 2.6 of \cite{CRS10}}\nonumber\\
&\leq&C\rho^n+C\bigg(\frac{r}{\rho}\bigg)^{n+1+\beta}\int_{B_\rho}|y|^\beta|\nabla u|^2.
\end{eqnarray}
Let now $\delta<1/2$. If $\rho=\delta^k$, $r=\delta^{k+1}$, $\mu=\delta^n$, then \eqref{2.4} gives
$$
\int_{B_{\delta^{k+1}}}|y|^\beta|\nabla u|^2\leq C\mu^k+C\mu\delta^{2(1-\alpha)}\int_{B_{\delta^k}}|y|^\beta|\nabla u|^2,
$$
which in turn, by choosing $\delta$ such that $q=C\delta^{2(1-\alpha)}<1$ and using induction, gives
$$
\int_{B_{\delta^k}}|y|^\beta|\nabla u|^2\leq\frac{C^2}{1-q}\mu^{k-1}.
$$
Therefore, for all $r<1/2$
\begin{equation}\label{2.5}
\int_{B_r}|y|^\beta|\nabla u|^2\leq Cr^n.
\end{equation}
Now, if $\alpha\leq1/2$, then $\beta\geq0$, and
$$
\int_{B_r}|\nabla u|\leq\bigg(\int_{B_r}|y|^{-\beta}\bigg)^{1/2}\bigg(\int_{B_r}|y|^\beta|\nabla u|^2\bigg)^{1/2}\leq Cr^{n+\alpha},
$$
which is \eqref{2.1} with $p=1$, and so $u\in C^{0,\alpha}(B_{1/2})$.

In case of $\alpha>1/2$, we get from \eqref{2.5}
$$
\int_{B_r}|\nabla u|^2\leq r^{-\beta}\int_{B_r}|y|^\beta|\nabla u|^2\leq Cr^{n-\beta}=Cr^{n-1+2\beta},
$$
which is \eqref{2.1} with $p=2$.
\end{proof}

Next we prove the non-degeneracy of a minimizer.

\begin{theorem}[Non-degeneracy] \label{t2.2}
If $u$ is a local minimizer of \eqref{1.7}, then there exists a constant $C_0>0$ such that for all $x\in B_{1/2}^n(0)\cap\{u>0\}$,
$$
u(x,0)\geq C_0\textrm{dist}(x,\partial\{u>0\})^\alpha.
$$
\end{theorem}
\begin{proof}
Assume that the conclusion of the theorem is not true. It means that for every $C_0>0$ there is a point $z\in B_{1/2}^n(0)\cap\{u>0\}$ such that  $u(z,0)<C_0\textrm{dist}(z,\partial\{u>0\})^\alpha$. We denote by $d$ the distance of the point $(z,0)$ from the free boundary. We also define $\delta:=u(z,0)$. The contradictory assumption implies that $\delta$ can be made as small as we wish.

From the Harnack inequality, \cite{FKS82}, there exists $c>0$ such that $u\leq c\delta$ in $B_{d}(z,0)$. Now if $\gamma$ is a smooth nonnegative function such that
$$
\gamma(x,y)=0\,\,\textrm{ in }\,\, B_{d/2}(z,0),\,\,\,\,\gamma(x,y)=2c\,\,\textrm{ in }\,\, B_{7d/8}(z,0)\setminus B_{3d/4}(z,0),
$$
we define
$$
v(x,y):=\min(u(x,y),\delta\gamma(x,y)).
$$
Note that $v\in H^1(\beta,B_{d}(z,0))$, and $v=u$ at the boundary of the ball, so it is an admissible test function for \eqref{1.7}. Therefore
\begin{equation}\label{2.6}
I_\varepsilon(u,B_{d}(z,0))\leq I_\varepsilon(v,B_{d}(z,0)).
\end{equation}
On the other hand, from the definition of $v$ one has
$$
\int_{B_{d}(z,0)}|y|^\beta|\nabla v|^2\leq\int_{B_{d}(z,0)}|y|^\beta|\nabla u|^2+O(\delta),
$$
and since $v=0$ in $B_{d/2}(z,0)$, then
$$
\mathcal{L}^n(\{v>0\})\leq\mathcal{L}^n(\{u>0\})-\mathcal{L}^n(B_{d/2}(z,0)\cap\mathbb{R}^n),
$$
which implies that
$$
f_\varepsilon\big(\mathcal{L}^n(\{v>0\})\big)<f_\varepsilon\big(\mathcal{L}^n(\{u>0\})\big).
$$
Therefore,
$$
I_\varepsilon(u,B_{d}(z,0))>I_\varepsilon(v,B_{d}(z,0)),
$$
which contradicts \eqref{2.6}. We remark that the constant $C_0$ may depend on $\varepsilon$.
\end{proof}

Now, arguing as in  \cite[Proposition 3.3]{CRS10}, it is possible to show that $u$ is in fact strongly non degenerate.

\begin{lemma}\label{l2.2}
If $u$ is a local minimizer of \eqref{1.7} in $B_1$, and $(0,0)$ is a free boundary point, then there is $C>0$ such that for $r\in(0,1/2)$,
$$
\sup_{B_r^n}u\geq Cr^\alpha.
$$
\end{lemma}

As a consequence of the H\"{o}lder regularity result and the non-degeneracy, we get the positive density result below.

\begin{theorem}[Positive density]\label{t2.3}
Let $(0,0)$ be a free boundary point. If $u$ is a local minimizer of \eqref{1.7} in $B_1$, then there is a constant $C_1>0$ such that for every $r>0$
\begin{equation}\label{2.7}
\mathcal{L}^n(\{u=0\}\cap B_r^n)\geq C_1r^n,\,\,\,\,\,\,\mathcal{L}^n(\{u>0\}\cap B_r^n)\geq C_1r^n.
\end{equation}
\end{theorem}

\begin{proof}
 In fact, from the non-degeneracy we know that there is $y\in B_r^n$ such that $u(y)\geq Cr^\alpha>0$. By H\"{o}lder continuity, $u>0$ in $B_{\delta r}^n(y)$ for a small $\delta>0$, which gives us the second estimate of \eqref{2.7}.

To prove the first estimate of \eqref{2.7}, it is enough to consider the case $r=1$. Assume the contrary. There is a sequence of minimizers $u_k$ defined in $B_1$, such that $0\in\partial\{u_k>0\}$ and
\begin{equation}\label{2.8}
\lim_{k\rightarrow+\infty}\mathcal{L}^n(\{u_k=0\})=0.
\end{equation}
Recall that $u_k$ is uniformly H\"{o}lder continuous. We may assume that the sequence converges to $u_0$ uniformly. Moreover, one has
$$
\int_{B_1}|y|^\beta|\nabla u_0|^2\,dx\,dy\leq\liminf_{k\rightarrow+\infty}\int_{B_1}|y|^\beta|\nabla u_k|^2\,dx\,dy.
$$
For every $v$ agreeing with $u_k$ on $\partial B_1$ one has $I_\varepsilon(u_k,B_1)\leq I_\varepsilon(v,B_1)$. Together with \eqref{2.8} this implies for every $v\in H^1(\beta,B_1)$ which agrees with $u_0$ on $\partial B_1$,
\begin{eqnarray*}
&&f_\varepsilon\big(\mathcal{L}^n(B_1^n)\big)+\int_{B_1}|y|^\beta|\nabla u_0|^2\,dx\,dy\leq I_\varepsilon(v,B_1)\nonumber\\
&&\leq f_\varepsilon\big(\mathcal{L}^n(B_1^n)\big)+\int_{B_1}|y|^\beta|\nabla v|^2\,dx\,dy.
\end{eqnarray*}
Therefore, $u_0$ minimizes the Dirichlet integral over the unit ball of $\mathbb{R}^n$, and as such, satisfies $\Div(|y|^\beta\nabla u_0)=0$ in $B_1$. Since $u_0(0)=0$ and $u_0\geq0$, then the strong maximum principle, see \cite{B97}, implies that $u_0\equiv0$, which is a contradiction on the non-degeneracy  property. Once again we remark, that the constant $C_1$ may depend on $\varepsilon$.
\end{proof}

\section{Further properties of solutions}\label{sct FPS}

The ultimate goal of this section is to prove that the free boundaries of local minimizers of \eqref{1.7} have local finite parameter. The results in this section are the analogue of the ones from \cite{AC81}.

\begin{proposition}\label{p3.1}
For a local minimizer $u$ in $\Omega$, $\mu (u):=-(-\Delta)^\alpha u$ is a nonnegative Radon measure with support in $\Omega\cap\partial\{u>0\}$.
\end{proposition}
\begin{proof}
Once again we recall the extension result from \cite{CS07} (see also Theorem \ref{t1.1} above). The proof is now the same as the one of Remark 4.2 of \cite{AC81}.
\end{proof}
In the spirit of \cite{AC81} a representation theorem can be proven. It plays an important role in the study of the free boundary.

\begin{theorem}[Representation theorem] \label{t3.1}
If $u$ is a local minimizer in $\Omega$, then
\begin{enumerate}
\item $\mathcal{H}^{n-1}(\mathcal{K}\cap\partial\{u>0\}\cap\mathbb{R}^n)<\infty$, for every compact set $\mathcal{K}\subset\Omega$.
\item There exists a Borel fucntion $q_\varepsilon$ such that
$$
\mu(u)=q_\varepsilon\mathcal{H}^{n-1}\lfloor\partial\{u>0\},
$$
that is for any $\psi\in C_0^\infty(\Omega)$ there holds
$$
-\int_\Omega|y|^\beta\nabla u\cdot\nabla\psi=\int_{\{u>0\}}\psi q_\varepsilon\,d\mathcal{H}^{n-1}.
$$
\item There exist constants $0<c<C<\infty$ depending on $n$, $\Omega$ and $\varepsilon$ such that for $B_r(x)\subset\Omega$ and $x\in\partial\{u>0\}$ one has $c\leq q_\varepsilon(x)\leq C$ and
$$
cr^{n-1}\leq\mathcal{H}^{n-1}\big(B_r(x)\cap\partial\{u>0\}\cap\mathbb{R}^n\big)\leq Cr^{n-1}.
$$
\item For $\mathcal{H}^{n-1}$ almost everywhere in $\partial\{u>0\}$ an outward normal $\nu=\nu(x_0)$ is defined and furthermore
$$
u(x_0+x)=q_\varepsilon(x_0)(x\cdot\nu(x_0))^\alpha_++o(|x|),\,\,\,\textrm{ as }x\rightarrow0.
$$
\item $\mathcal{H}^{n-1}\big((\partial\{u>0\}\cap\mathbb{R}^n)\setminus(\partial^*\{u>0\}\cap\mathbb{R}^n)\big)=0$,
where $\partial^*$ is the reduced boundary.

\item The reduced free boundary $\partial^*\{u>0\}\cap\mathbb{R}^n$ is locally a $C^{1,\gamma}$ surface.
\end{enumerate}
\end{theorem}
\begin{proof}
The first three assertions of the theorem follow as those of Theorem 4.5 of \cite{AC81}. For the first assertion we also refer the reader to Theorem 1.1 of \cite{SS12}. Note that $\Omega\cap\{u>0\}\cap\mathbb{R}^n$ has finite perimeter, thus, the reduced boundary is defined as well as the measure theoretic normal $\nu(x)$, for $x\in\partial^*\{u>0\}$ (see, for example \cite{EG92}). The 5th assertion of the theorem is a consequence of the 3rd one and properties of minimizers proved above (see \cite{EG92}). $C^{1,\gamma}$ regularity of the reduced free boundary follows by \cite{SSS}, see also \cite{SR}. \\

The proof of 4 is similar to the corresponding one in \cite{DP05} (Theorem 5.5), but since in our case we are dealing with the fractional Laplacian, some modifications need to be done (similar to \cite{CRS10}).

For a minimizer $u$ we will denote by $\Omega_-(u)\subset\mathbb{R}^n$ the set where it is 0, and by $\Omega_+(u)\subset\mathbb{R}^n$ its positivity set. We will also use $\Gamma(u)\subset\mathbb{R}^n$ to denote the free boundary of $u$, and $\Gamma^*(u)$ - the reduced free boundary.

The reduced free boundary is the set of points $x_0$ at which the following holds (see \cite{G84}): given the half ball $(B_r^n)^+(x_0):=\{(x-x_0)\cdot\nu\geq0\}\cap B_r^n(x_0)$, one has
$$
\lim_{r\rightarrow0}\frac{\mathcal{L}^n\big((B_r^n)^+(x_0)\triangle\Omega_+(u)\big)}{\mathcal{L}^n\big(B_r^n(x_0)\big)}=0.
$$
Note that from the uniform density of $\Omega^\pm$ one has, as $r\rightarrow0$ at the free boundary point $x_0$
\begin{equation}\label{3.1}
B_r^n(x_0)\cap\Gamma^*(u)\subset\{|(x-x_0)\cdot\nu(x_0)|\leq o(r)\}.
\end{equation}

Indeed, if $u(x)=0$ for $(x-x_0)\cdot\nu(x_0)\geq\delta r$, there is $\gamma>0$ such that $\mathcal{L}^n\big(B_{\delta r}^n(x)\cap\{u=0\}\big)\geq\gamma\delta r^n$, which implies
$$
\liminf_{r\rightarrow0}\frac{\mathcal{L}^n\big((B_r^n)^+(x_0)\triangle\Omega_+(u)\big)}{\mathcal{L}^n\big(B_r^n(x_0)\big)}\geq\gamma\delta;
$$
a contradiction.

The same argument is valid, if $x\in\Omega_-(u)$ is such that $(x-x_0)\cdot\nu\leq-\delta r$.\\

In order to prove 4, we remark, that in fact it follows from the fact that blow-up limits at regular points are one-dimensional. To prove it, without loss of generality we assume that $\nu(x_0)=e_n$. Let
$$
u_r(x,y):=\frac{1}{r^\alpha}u(x_0+rx,ry)
$$
be a blow-up of $u$, and let $u_0$ be a blow-up limit. We need to prove that $u_0(x,0)=q(x_n)^\alpha_+$, where $q$ is a constant. There exists a coordinate system $(x',x_n)$ centered at $0$ such that
\begin{itemize}
\item $\Omega_+(u_0)=\mathbb{R}^n_+$ (this is because of \eqref{3.1}),
\item $(-\Delta)^\alpha u_0=0$ in $\Omega_+(u)$.
\end{itemize}
Define $v(x)=(x_n)^\alpha_+$. By optimal regularity and non-degeneracy, there are positive constants $C_1$, $C_2$ such that $0<C_1v\leq u_0\leq C_2v$. By applying the oscillation lemma (see \cite{FKS82} or Theorem 2.5 of \cite{CRS10}), one has that there exists $\lambda\in(0,1)$  such that for all small enough $r>0$,
$$
\osc_{B_r}\frac{u_0}{v}\leq\lambda\osc_{B_{2r}}\frac{u_0}{v}.
$$
Since Harnack constants are invariant under the blow-up scaling, the oscillation lemma holds at every scale, the solutions being global. Thus, one may apply it all the way down from a ball of radius $2^Nr$ ($N$ being arbitrary large) to a ball of radius $r$. Therefore, $\frac{u_0}{v}$ is a constant. Of course, that constant depends on $\varepsilon$ and also on the blow-up point. Although for our further work this fact is not very important, but by following the argument of \cite{AC81}, which proves 2nd assertion of the theorem, one can see that the constant is actually $q_\varepsilon$ appearing in 2.
\end{proof}
Next, we see that points, at which the free boundary has a tangent ball, are regular points.
\begin{definition}\label{d3.1}
We say that a free boundary point $x_0\in\Gamma(u)$ has a tangent ball from outside, if there is a ball $B\subset\Omega_-(u)$ such that $x_0\in B\cap\Gamma(u)$. A point $x_0\in\Gamma(u)$ is said to be regular, if $\Gamma(u)$ has a tangent hyperplane at $x_0$.
\end{definition}

The following result is from \cite{CRS10}. The proof can also be concluded from the 4th assertion of the Theorem \ref{t3.1} together with Theorem \ref{t4.1} below.
\begin{theorem}(The free boundary condition).\label{t3.2}
Let $x_0\in\Gamma(u)$ be a regular point. There exists a constant $\lambda_\varepsilon(\alpha)$ such that
$$
\lim_{x\rightarrow x_0}\frac{u(x)}{\big((x-x_0)\cdot\nu(x_0)\big)_+^\alpha}=\lambda_\varepsilon(\alpha).
$$
\end{theorem}

\section{A fractional Hadamard formula}\label{sct HF}

In this section we proof a fractional Hadamard's variational formula, which will provide a rate control upon volume decreasing versus energy increasing, for competing shapes.

Let $u$ be a local minimizer. We recall the notation $\Gamma^*(u)$ for its reduced free boundary. For given two points $x_1$, $x_2\in\Gamma^*(u)$, the idea is to make an inward perturbation around $x_1$, and outward perturbation around $x_2$ in such a way, that we do not change very much the original volume, and then compare the optimal configuration to the perturbated one in terms of the functional $I_\varepsilon$. We proceed as follows.

Let $\rho:\mathbb{R}\rightarrow\mathbb{R}$ be a nonnegative function from $C_0^\infty[0,1]$ with $\int\rho=1$. For any $r\in(0,\frac{dist(x_1,x_2)}{100})$ and $\gamma>0$, we define

\[ P_r(x,y):= \left\{
  \begin{array}{l l}
    (x,y)+\gamma r\rho\big(\frac{|(x-x_1,y)|}{r}\big)\nu(x_1,0), & \quad \text{when $(x,y)\in B_r(x_1,0)$}\\
    (x,y)-\gamma r\rho\big(\frac{|(x-x_2,y)|}{r}\big)\nu(x_2,0), & \quad \text{when $(x,y)\in B_r(x_2,0)$}\\
    (x,y) & \quad \textrm{elsewhere}.
  \end{array} \right.\]

If $v$ is any vector in $\mathbb{R}^{n+1}$, direct computations show that in $B_r(x_i,0)$
\begin{equation}\label{4.1}
DP_r(x,y)\cdot v=v+(-1)^{i+1}\bigg\{\gamma\rho'\bigg(\frac{|(x-x_i,y)|}{r}\bigg)\frac{\langle (x-x_i,y),v\rangle}{|(x-x_i,y)|}\bigg\}\nu(x_i,0),
\end{equation}
where $\langle\cdot,\cdot\rangle$ the the inner product in $\mathbb{R}^{n+1}$ and $i$ takes the values 1 and 2. Note, that if $\gamma$ is small enough, then $P_r$ is a diffeomorphism that maps $B_r(x_i,0)$ onto itself. Indeed, if $\displaystyle{\gamma\sup_{t\in[0,1]}}\rho'(t)<1$, then $P_r$ is a local injective diffeomorphism. Now, if $\gamma\rho(t)\leq1-t$, for $t\in[0,1]$,
$$
|P_r(x,y)-(x_i,0)|\leq|(x-x_i,y)|+\gamma r\rho\bigg(\frac{|(x-x_i,y)|}{r}\bigg)\leq r,
$$
for any $(x,y)\in B_r(x_i,0)$. Finally, note that $P_r=Id$ on $\partial B_r(x_i,0)$, therefore $P_r$ has to be onto.

For each $r>0$ small enough, we consider the $r$-perturbed configuration by
$$
v_r(P_r(x,y))=u(x,y).
$$
The idea is to compare our optimal configuration $\{u>0\}$ to its perturbation $\{v_r>0\}$ in terms of the penalized problem \eqref{1.7}. For any $r>0$ small enough and $i=1,2$, we consider the blow-up sequence $u_r^i:\,B_1(0)\rightarrow\mathbb{R}$ given by
$$
u_r^i(x,y):=\frac{1}{r^\alpha}u(x_i+rx,ry).
$$
From the blow-up analysis (see \cite{DP05}) we know that the set $B_1\cap\{u_r^i>0\}\cap\mathbb{R}^{n}$ approaches to
$\{(x,0)\in B_1,\,\,\langle(x,0)\cdot\nu(x_i,0)\rangle<0\}$, as $r\rightarrow0$.
In order to compute the change on the volume of the perturbation, we use the Change of Variables Theorem to obtain
\begin{eqnarray}\label{4.2}
&&\frac{\mathcal{L}^n\big(B_r(x_i,0)\cap\{v_r>0\}\cap\mathbb{R}^{n}\big)}{r^{n}}=\frac{1}{r^{n}}\int_{B_r(x_i,0)\cap\{v_r>0\}\cap\mathbb{R}^{n}}dz\nonumber\\
&=&\int_{B_1\cap\{v_r(x_i+rx,ry)>0\}\cap\mathbb{R}^{n}}dx\,dy\nonumber\\
&=&\int_{B_1\cap\{u_r^i(x_i+rx,ry)>0\}\cap\mathbb{R}^{n}}\det(DP_r(x_i+rx,ry))\,dx\,dy\\
&\rightarrow&\int_{B_1\cap\{\langle (x,0),\nu(x_i,0)\rangle<0\}\cap\mathbb{R}^n}1+(-1)^{i+1}\gamma\rho'(|(x,y)|)\bigg\langle\frac{(x,y)}{|(x,y)|},\nu(x_i,0)\bigg\rangle\,dx\,dy,\nonumber
\end{eqnarray}
as $r\rightarrow0$.
Note that there is a constant $C(\rho)$ such that for any unit vector $\nu\in\mathbb{S}^{n}$
\begin{equation}\label{4.3}
C(\rho)=\int_{B_1\cap\{\langle (x,0),\nu\rangle<0\}\cap\mathbb{R}^n}\rho'(|(x,y)|)\bigg\langle\frac{(x,y)}{|(x,y)|},\nu\bigg\rangle\,dx\,dy.
\end{equation}
A similar computation shows that
\begin{equation}\label{4.4}
\frac{\mathcal{L}^n\big(B_r(x_i,0)\cap\{u>0\}\cap\mathbb{R}^n\big)}{r^{n}}\rightarrow\int_{B_1\cap\{\langle (x,0)\cdot\nu(x_i,0)<0\rangle\}\cap\mathbb{R}^n}dx,
\end{equation}
as $r\rightarrow0$. From \eqref{4.2}, \eqref{4.3} and \eqref{4.4} we get
$$
\frac{1}{r^{n}}\bigg[\mathcal{L}^n\big(B_r(x_i,0)\cap
\{v_r>0\}\cap\mathbb{R}^n\big)-\mathcal{L}^n\big(B_r(x_i,0)
\cap\{u>0\}\cap\mathbb{R}^n\big)\bigg]\rightarrow0,
$$
as $r\rightarrow0$.
From the Lipschitz continuity of the penalization $f_\varepsilon$, we obtain
\begin{eqnarray}\label{4.5}
f_\varepsilon\big(\mathcal{L}^n\big(B_r(x_i,0)\cap\{v_r>0\}\cap\mathbb{R}^n\big)\big)&-&f_\varepsilon\big(\mathcal{L}^n\big(B_r(x_i,0)\cap\{u>0\}\cap\mathbb{R}^n\big)\big)\nonumber\\
&\leq&\frac{1}{\varepsilon}o(r^{n}).
\end{eqnarray}
Now we should check what happens with the integral part of the functional. Initially we observe that
\begin{eqnarray}\label{4.6}
&&\frac{1}{r^{n}}\int_{B_r(x_i,0)}|y|^\beta|\nabla u(x,y)|^2\,dx\,dy=\int_{B_1}|y|^\beta|\nabla u_r^i(x,y)|^2\,dx\,dy\nonumber\\
&=&\int_{B_1\cap\{u_r^i>0\}}|y|^\beta|\nabla u_r^i(x,y)|^2\,dx\,dy,
\end{eqnarray}
once $\Gamma^*(u_r^i)$ is smooth. Next we apply twice the Change of Variables Theorem and take into account that $P_r$ maps $B_r(x_i,0)$ diffeomorphically onto itself to write
\begin{eqnarray}\label{4.7}
&&\frac{1}{r^{n}}\int_{B_r(x_i,0)}|y|^\beta|\nabla v_r(x,y)|^2\,dx\,dy\nonumber\\
&&=\frac{1}{r^{n}}\int_{B_r(x_i,0)}|y|^\beta|DP_r(P_r^{-1}(x,y))^{-1}\cdot\nabla u(P_r^{-1}(x,y)|^2\,dx\,dy\nonumber\\
&&=\frac{1}{r^{n}}\int_{B_r(x_i,0)}|y|^\beta|DP_r(z,y)^{-1}\cdot\nabla u(z,y)|^2|\det\big(DP_r(z,y)\big)|\,dz\,dy\\
&&=\int_{B_1\cap\{u_r^i>0\}}|y|^\beta|DP_r(x_i+rh,ry)^{-1}\cdot\nabla u_r^i(h,y)|^2|\det\big(DP_r(x_i+rh,ry)\big)|\,dh\,dy\nonumber.
\end{eqnarray}
From \eqref{4.1}, using the fact that for any matrix $A$ with $|A|<1$,
$$
(Id+A)^{-1}=Id+\sum_{i=1}^\infty(-1)^iA^i,
$$
we have
\begin{eqnarray}\label{4.8}
&&DP_r(x_i+rh,ry)^{-1}\cdot\nabla u_r^i(h,y)\\
&=&\nabla u_r^i(h,y)-(-1)^{i+1}\gamma\frac{\rho'(|(h,y)|)}{|(h,y)|}\langle (h,y),\nabla u_r^i(h,y)\rangle\nu(x_i,0)+o(\gamma)\nonumber.
\end{eqnarray}
On the other hand,
\begin{equation}\label{4.9}
|\det\big(DP_r(x_i+rh,ry)\big)|=1+(-1)^{i+1}\gamma\frac{\rho'(|(h,y)|)}{|(h,y)|}\langle (h,y),\nu(x_i,0)\rangle.
\end{equation}
Combining \eqref{4.6}, \eqref{4.7}, \eqref{4.8} and \eqref{4.9}, we obtain
\begin{eqnarray}\label{4.10}
&&\frac{1}{r^{n}}\int_{B_r(x_i,0)}|y|^\beta\big[|\nabla v_r(x,y)|^2-|\nabla u(x,y)|^2\big]\,dx\,dy\nonumber\\
&=&(-1)^{i+1}\gamma\int_{B_1\cap\{u_r^i>0\}}|y|^\beta|\nabla u_r^i(h,y)|^2\frac{\rho'(|(h,y)|)}{|(h,y)|}\langle (h,y),\nu(x_i,0)\rangle\,dh\,dy\\
&+&(-1)^i2\gamma\int_{B_1\cap\{u_r^i>0\}}|y|^\beta\frac{\rho'(|(h,y)|)}{|(h,y)|}\langle (h,y),\nabla u_r^i(h,y)\rangle\langle \nabla u_r^i(h,y),\nu(x_i,0)\rangle\,dh\,dy\nonumber\\
&+&o(\gamma).\nonumber
\end{eqnarray}
Now we recall that in the proof of 4 of Theorem \ref{t3.1}, we verified that $u_r^i(h,y)=q_\varepsilon(x_i)(h_n)_+^\alpha+o(r)$, as $r\rightarrow0$, i.e. blow-up limits are one dimensional. Therefore, from the blow-up analysis, we have, as $r\rightarrow0$
$$
\int_{B_1\cap\{u_r^i>0\}}|y|^\beta|\nabla u_r^i|^2\rightarrow\alpha^2q_\varepsilon^2(x_i)\int_{B_1\cap\{\langle z,\nu_i\rangle<0\}}|y|^\beta|h_n|^{2(\alpha-1)}\,dh\,dy.
$$
Hence, letting $r\rightarrow0$ in \eqref{4.10} leads to
\begin{equation}\label{4.11}
\frac{1}{r^{n}}\int_{B_r(x_i,0)}|y|^\beta\big[|\nabla v_r|^2-|\nabla u|^2\big]\rightarrow(-1)^{i+1}\gamma\alpha^2q_\varepsilon^2(x_i)c(\rho)+o(\gamma),
\end{equation}
where
$$
c(\rho):=\lim_{r\rightarrow0}\int_{B_1\cap\{u_r^i>0\}}\frac{|y|^\beta}{|h_n|^{1+\beta}}\frac{\rho'(|(h,y)|)}{|(h,y)|}\langle (h,y),\nu(x_i,0)\rangle\,dh\,dy
$$
is a positive constant. To check that indeed, $0<c(\rho)<\infty$ is a positive constant, we argue as follows: since
$$
\Div\big(\rho(|z|)\big)=\frac{\rho'(|z|)}{|z|}\langle z,\nu(x_i,0)\rangle,
$$
so the divergence theorem together with the blow-up analysis provides
$$
\int_{B_1\cap\{u_r^i>0\}}\frac{\rho'(|z|)}{|z|}\langle z,\nu(x_i,0)\rangle\,dz\rightarrow\int_{B_1\cap\{\langle z,\nu(x_i,0)\rangle=0\}}\rho(|z|)\,d\mathcal{H}^{n-1}(z)=\text{const.}
$$
Recalling that the function $\rho$ is compactly supported in $[0,1]$, we conclude, that $0<c(\rho)<\infty$.  Returning to \eqref{4.11}, we can write
\begin{eqnarray}\label{4.110}
\int\limits_\Omega|y|^\beta[|\nabla v_r|^2-|\nabla u|^2]\,dx\,dy =r^n\gamma\alpha^2c(\rho)\big(q_\varepsilon^2(x_1)-q_\varepsilon^2(x_2)\big)+r^no(\gamma).
\end{eqnarray}
Combining \eqref{4.110} with the minimality property of $u$, we get
\begin{eqnarray}\label{4.12}
0\leq I_\varepsilon(v_r)-I_\varepsilon(u)&\leq& r^{n}\gamma\alpha^2 c(\rho)\big(q_\varepsilon^2(x_1)-q_\varepsilon^2(x_2)\big)\nonumber\\
&+&r^{n}o(\gamma)+\frac{1}{\varepsilon}o(r^{n}),
\end{eqnarray}
which gives after dividing by $r^{n}$ and letting $r\rightarrow0$
$$
0\leq\gamma\alpha^2 c(\rho)\big(q_\varepsilon^2(x_1)-q_\varepsilon^2(x_2)\big)+o(\gamma).
$$
Now dividing by $\gamma$ and letting $\gamma\rightarrow0$, and afterwards reversing the places of $x_1$ and $x_2$, we obtain
$$
q_\varepsilon(x_1)=q_\varepsilon(x_2).
$$
Since $x_1$ and $x_2$ were arbitrary in $\Gamma^*(u)$, we actually proved
\begin{theorem}\label{t4.1}
On the reduced free boundary, we have $q_\varepsilon\equiv\lambda_\varepsilon(\alpha)$.
\end{theorem}
Note also that \eqref{4.10} provides the Hadamard's variational formula:
\begin{equation}\label{4.13}
\int_\Omega|y|^\beta\big[|\nabla v_r|^2-|\nabla u|^2\big]=\lambda_\varepsilon^2(\alpha)V+o(V),
\end{equation}
where $V$ is the volume change.

\section{Recovering the original problem}\label{sct Rec}

Here we shall relate a solution to the penalized problem to a (possible) solution of our original problem. The idea is that the function $f_\varepsilon$ will charge a lot for those configurations that have a volume bigger than $\omega$. We hope that if the charge is too big optimal configuration of problem \eqref{1.7} will prefer to have volume $\omega$ than pay for the penalization.
\begin{proposition}\label{p5.1}
There exist $C,c>0$ constants such that
$$
0<c\leq\mathcal{L}^n(\{u_\varepsilon>0\}\cap\mathbb{R}^n\cap\Omega)\leq\omega+C\varepsilon,
$$
where $u_\varepsilon$ is a solution to \eqref{1.7}.
\end{proposition}
\begin{proof}
Let $\Omega_*$ be any smooth domain such that its complement contains $\Omega^c$ with $\mathcal{L}^n((\Omega_*^c\cap\mathbb{R}^n)\setminus(\Omega^c\cap\mathbb{R}^n))=\omega$.
From the minimality of $u_\varepsilon$ we have
\begin{equation}\label{5.1}
I_\varepsilon(u_\varepsilon,\Omega)=\int_\Omega|y|^\beta|\nabla u_\varepsilon|^2+f_\varepsilon(\mathcal{L}^n(\{u_\varepsilon>0\}\cap\mathbb{R}^n\cap\Omega))\leq I_\varepsilon(u_*,\Omega_*)=C,
\end{equation}
where $u_*$ is the $\alpha$-harmonic function in $\Omega_*^c\setminus\Omega^c$ taking "boundary data" equal to $\varphi$ in $\Omega^c$ and $0$ on $\partial\Omega_*^c$. Thus
$$
\frac{1}{\varepsilon}(\mathcal{L}^n(\{u_\varepsilon>0\}\cap\mathbb{R}^n\cap\Omega)-\omega)\leq f_\varepsilon(\mathcal{L}^n(\{u_\varepsilon>0\}\cap\mathbb{R}^n\cap\Omega)\leq C.
$$
This proves the estimate from above. In order to prove the estimate from below, we first note that since the weight $|y|^\beta$ is in the second Muckenhoupt class $A_2$ for $\beta\in(-1,1)$, we have a Poincar\'{e} inequality (see, for example, \cite{CRS10}), which together with \eqref{5.1} provides
\begin{equation}\label{5.2}
\int_\Omega|y|^\beta[|\nabla u_\varepsilon|^2+|u_\varepsilon|^2]\leq C,
\end{equation}
for some $C$ independent of $\varepsilon$. Recalling the fact that $u_\varepsilon$ takes values $\varphi$ outside of the domain, where it is $\alpha$-harmonic, recalling also the mean value inequality and \eqref{5.2}, we obtain (by integrating along layers with $E:=\Omega\cap B_\delta(\partial\Omega)$)
$$
\bigg(\int_{\partial\Omega}\varphi\bigg)^2\leq C(\delta)\mathcal{L}^n(\{u_\varepsilon>0\}\cap\mathbb{R}^n\cap E)\int_E|y|^\beta[|\nabla u_\varepsilon|^2+|u_\varepsilon|^2]
$$
the last integral being bounded uniformly in $\varepsilon$. Hence, the estimate from below is proved.
\end{proof}

\begin{lemma}\label{l5.1}
There exists $C>0$ depending on the domain and $\varphi$, but independent of $\varepsilon$, such that $\lambda_\varepsilon(\alpha)\leq C$.
\end{lemma}
The proof of this lemma is a consequence of Proposition \ref{p5.1}, isoperimetric inequality and 2nd assertion of the Theorem \ref{t3.1} (see the proof of the corresponding result in \cite{AAC86}, \cite{BMW06}, \cite{DP05}, or \cite{T10}).
\begin{lemma}\label{l5.2}
There exists $c>0$ depending on the domain and $\varphi$ but independent of $\varepsilon$ such that $c\leq\lambda_\varepsilon(\alpha)$.
\end{lemma}
\begin{proof}
As in the above mentioned references, the proof is based on the Hopf's Lemma, which is true also for the fractional Laplacian case (see \cite{B97} or Proposition 2.7 in \cite{CRS10}).

In fact, let $z_1\in\Omega$ be such that $u_\varepsilon(z_1)>0$ for all $\varepsilon>0$. Let $\delta:=dist(z_1,\partial\Omega)$. We then consider the smooth family of domains $\Gamma_t:=B_{\frac{\delta}{2}+t}(z_1)\cap\Omega$. Let $t_\varepsilon$ be the first $t$ such that $\Gamma_t$ touches $\partial\{u_\varepsilon>0\}$. Let $x_0$ be that touching point. Define $\psi_\varepsilon$ to be $\alpha$-harmonic in $\Gamma_{t_\varepsilon}\setminus\Gamma_0$, with the following "boundary data":
$$
\psi_\varepsilon|_{\partial\Gamma_0}=\min\varphi\,\,\,\,\textrm{ and }\,\,\,\,\psi_\varepsilon|_{\Gamma_{t_\varepsilon}^c}=0.
$$
By the maximum principle we have $u_\varepsilon\geq\psi_\varepsilon$ in $\Gamma_{t_\varepsilon}\setminus\Gamma_0$. From the (generalized) Hopf's Lemma we also know that there exists a constant $c>0$ depending only on the domain and $\varphi$, but independent of $\varepsilon$, such that
\begin{equation}\label{5.3}
\psi_\varepsilon(x)\geq c((x-x_0)\cdot\nu(x_0))^\alpha.
\end{equation}
On the other hand we have the following asymptotic development around $x_0$
\begin{equation}\label{5.4}
\psi_\varepsilon(x)\leq u_\varepsilon(x)=\lambda_\varepsilon(\alpha)((x-x_0)\cdot\nu(x_0))^\alpha+o(|x-x_0|).
\end{equation}
Letting $x\rightarrow x_0$ in \eqref{5.4} and taking into account \eqref{5.3}, we obtain
$$
c\leq\lambda_\varepsilon(\alpha)
$$
as desired.
\end{proof}
Now we are ready to prove the main theorem of this section.
\begin{theorem}\label{t5.1}
If $\varepsilon$ is small enough, then any solution of \eqref{1.7} is a solution of $(1.1)$.
\end{theorem}
\begin{proof}
Basically, we just need to show that $S:=\mathcal{L}^n(\{u_\varepsilon>0\}\cap\mathbb{R}^n\cap\Omega)=\omega$, for $\varepsilon$ small enough.

Suppose $S>\omega$. In the spirit of the previous section, we consider an inward perturbation of the positivity set of $u_\varepsilon$ with the volume change $V$, such that for the new function $\tilde{u}_\varepsilon$ we still have $\mathcal{L}^n(\{\tilde{u}_\varepsilon>0\}\cap\mathbb{R}^n\cap\Omega)>\omega$. Thus
\begin{equation}\label{5.5}
f_\varepsilon\big(\mathcal{L}^n(\{\tilde{u}_\varepsilon>0\}\cap\mathbb{R}^n\cap\Omega)\big)
-f_\varepsilon\big(\mathcal{L}^n(\{u_\varepsilon>0\}\cap\mathbb{R}^n\cap\Omega)\big)=-\frac{1}{\varepsilon}V.
\end{equation}
From Theorem \ref{t4.1} and Lemma \ref{l5.1} we have
\begin{eqnarray}\label{5.6}
\int_\Omega|y|^\beta[|\nabla\tilde{u}_\varepsilon|^2-|\nabla u_\varepsilon|^2]&=&\lambda_\varepsilon^2(\alpha)V+o(V)\nonumber\\
&\leq&C^2V+o(V).
\end{eqnarray}
Using the fact that $I_\varepsilon(u_\varepsilon,\Omega)\leq I_\varepsilon(\tilde{u}_\varepsilon,\Omega)$, from \eqref{5.5} and \eqref{5.6} we get
$$
0\leq C^2V+o(V)-\frac{1}{\varepsilon}V,
$$
therefore (by dividing on $V$ and letting $V\rightarrow0$) it provides us  with
$$
\varepsilon\geq\frac{1}{C^2},
$$
which is a contradiction, when $\varepsilon$ is small enough. Hence $S\leq\omega$, for small $\varepsilon$. If $S<\omega$, arguing the same way and using Lemma \ref{l5.2}, we obtain another lower bound for $\varepsilon$. Thus, when $\varepsilon$ is small enough, $S=\omega$.
\end{proof}

\end{document}